\documentclass[onefignum,onetabnum]{siamart190516}

\usepackage{lipsum}
\usepackage{amsfonts}
\usepackage{graphicx}
\usepackage{epstopdf}
\usepackage{amsmath,amssymb,epsfig,graphics,psfrag,latexsym,amsmath,listings,url,tabularx,gensymb}%%%
\usepackage{stmaryrd}
\usepackage[utf8]{inputenc}
\usepackage{multicol}
\usepackage{bbm}
\usepackage{color}
\usepackage{caption}
\usepackage{subcaption}
\usepackage{graphicx,xcolor} 
\usepackage[framemethod=tikz]{mdframed}%%%
\Crefname{ALC@unique}{Line}{Lines} % <- Preamble
\usepackage{algorithmic}
\ifpdf
  \DeclareGraphicsExtensions{.eps,.pdf,.png,.jpg}
\else
  \DeclareGraphicsExtensions{.eps}
\fi

\usepackage{siunitx} % see https://tex.stackexchange.com/a/458965 and http://ctan.math.washington.edu/tex-archive/macros/latex/contrib/siunitx/siunitx.pdf
\usepackage{booktabs}

\usepackage{rotating}

\newsiamremark{remark}{Remark}
\newsiamremark{hypothesis}{Hypothesis}
\crefname{hypothesis}{Hypothesis}{Hypotheses}
\newsiamthm{claim}{Claim}

\headers{Multilateration of Random Networks with Community Structure}{R. C. Tillquist and M. E. Lladser}

\title{Multilateration of Random Networks with Community Structure \thanks{Submitted to the editors DATE.
\funding{This work has been partially funded by the NSF IIS grant No. 1836914.}}}

\author{Richard C. Tillquist\thanks{Department of Computer Science, University of Colorado, Boulder.}
\and Manuel E. Lladser\thanks{Department of Applied Mathematics, University of Colorado, Boulder. Corresponding author.} \email{(manuel.lladser@colorado.edu)}
}

\usepackage{amsopn}
\DeclareMathOperator{\diag}{diag}

\ifpdf
\hypersetup{
  pdftitle={Multilateration of Random Networks with Community Structure},
  pdfauthor={R.C. Tillquist and M.E. lladser}
}
\fi

\usepackage[colorinlistoftodos,textsize=tiny]{todonotes}
\newcommand{\Comments}{1}
\newcommand{\mynote}[2]{\ifnum\Comments=1\textcolor{#1}{#2}\fi}
\newcommand{\mytodo}[2]{\ifnum\Comments=1%
  \todo[linecolor=#1!80!black,backgroundcolor=#1,bordercolor=#1!80!black]{#2}\fi}

\newcommand{\E}{\mathbb{E}}
\newcommand{\K}{\mathcal{K}}
\newcommand{\NN}{\mathbb{N}}
\newcommand{\kbf}{\mathbf{k}}

\newcommand{\Prob}{\mathbb{P}}

\newcommand{\SBM}{{\hbox{SBM}}}

\newcommand{\xbf}{\mathbf{x}}

\newcommand{\ybf}{\mathbf{y}}

\def\Abf{{\mathbf{A}}}

\def\Dbf{{\mathbf{D}}}
\def\diam{{\hbox{diam}}}

\def\Mbf{{\mathbf{M}}}

\begin{document}

\maketitle

\begin{abstract}
The minimal number of nodes required to multilaterate a network endowed with geodesic distance (i.e., to uniquely identify all nodes based on shortest path distances to the selected nodes) is called its \emph{metric dimension}. This quantity is related to a useful technique for embedding graphs in low-dimensional Euclidean spaces and representing the nodes of a graph numerically for downstream analyses such as vertex classification via machine learning. While metric dimension has been studied for many kinds of graphs, its behavior on the Stochastic Block Model (SBM) ensemble has not. The simple community structure of graphs in this ensemble make them interesting in a variety of contexts. Here we derive probabilistic bounds for the metric dimension of random graphs generated according to the SBM, and describe algorithms of varying complexity to find---with high probability---subsets of nodes for multilateration. Our methods are tested on SBM ensembles with parameters extracted from real-world networks. We show that our methods scale well with increasing network size as compared to the state-of-the-art Information Content Heuristic algorithm for metric dimension approximation.
\end{abstract}

\begin{keywords}
Erd\"os-R\'enyi model, graph embedding, metric dimension, multilateration, resolving set, stochastic block model
\end{keywords}

\begin{AMS}
05C80, 68R10, 68W25, 90B15
\end{AMS}

\section{Introduction}

In the Euclidean plane any set of three affinely independent points is enough to perform \emph{trilateration}; in particular, every point in the plane can be uniquely identified by its distances to such a set. This idea may be extended to arbitrary metric spaces, and specifically to nodes in a graph, as follows.

Let $G=(V,E)$ be a graph. Throughout this manuscript we think of $\Dbf$ as the \textit{distance matrix} and of $\Abf$ as the \textit{adjacency matrix} of $G$ and denote the \textit{diameter} of $G$ as $\diam(G)$. A set $R \subseteq V$ is called \emph{resolving} if for all $u,v \in V$, with $u\ne v$, there is at least one $r \in R$ such that $\Dbf(u,r) \neq \Dbf(v,r)$. In this case, $R$ is said to \textit{multilaterate} or \textit{resolve} $G$ as distances to these nodes are enough to uniquely distinguish all nodes. 

The ability to resolve graphs in this way is relevant in a variety of settings, including robot navigation~\cite{khuller1996landmarks}, network discovery and verification~\cite{beerliova2006network}, the Mastermind game~\cite{chvatal1983mastermind}, and diffusions over networks~\cite{spinelli2016observer}. More recently, graph embeddings based on resolving sets have been applied to representing biological sequence data in a way amenable to machine learning classifiers~\cite{tillquist2019low}. Indeed, if $R$ resolves $G$ then any node in $G$ may be embedded into the $|R|$-dimensional Euclidean space using the map $\Phi_R(v):=(\Dbf(v,r))_{r\in R}$, for $v\in V$. Finding small resolving sets is therefore of great interest for a concise numerical representation of nodes in a graph, which, in turn, improves the efficiency of algorithms exercising this representation.

The size of a minimal resolving set in a graph $G$, denoted $\beta(G)$, is called its \emph{metric dimension}~\cite{harary1976metric, slater1975leaves}. Determining the metric dimension of arbitrary graphs is an NP-complete problem~\cite{gary1979computers,khuller1996landmarks}. Nevertheless, exact formulae or bounds are known for a variety of graph families~\cite{chartrand2000resolvability, harary1976metric, slater1975leaves, tillquist2019low} as well as for random graph ensembles, particularly random forests~\cite{mitsche2015limiting} and the Erd\"os-R\'enyi model~\cite{babai1980random, bollobas2012metric}. While some of these results are constructive, it is often the case that accurate estimates of metric dimension can be achieved without exhibiting resolving sets explicitly. In particular, these works do not necessarily elucidate how to multilaterate a given graph.

A number of approximation and heuristic algorithms have been developed for finding resolving sets---though not necessarily minimal---in general graphs~\cite{kratica2009computing,mladenovic2012variable}. The Information Content Heuristic (ICH) is among the most widely used of these methods as it guarantees an approximation ratio of $1 + (1 + o(1))\cdot\ln|V|$, the best possible for metric dimension~\cite{hauptmann2012approximation}. Nevertheless, with a time complexity of $O(|V|^3)$ for known $\Dbf$, this algorithm is prohibitively expensive when applied to large networks. This motivates us to focus on networks generated according to specific random graph models. By taking advantage of highly likely graph properties in these networks, it may be possible to design approximation algorithms with improved average efficiency.

The Stochastic Block Model (SBM) is a generative graph model that serves as the fundamental method by which graphs with simple community structure are modeled. It has been studied extensively with respect to community detection and recovery both theoretically~\cite{abbe2017community,abbe2016exact,abbe2015community,karrer2011stochastic}, and with respect to social interactions~\cite{carley2005special,traud2012social}, gene expression~\cite{cline2007integration,jiang2004cluster}, and recommender systems~\cite{linden2003amazon,sahebi1997community}.

Let $V=\{1,\ldots,n\}$ and $C$ be a partition of $V$ into $c$ subsets $V_1, \cdots, V_c$, representing communities. In this manuscript, we assume $c\ge1$ to be fixed and finite. Let $P$ be a $(c \times c)$ symmetric matrix with entries $0 \leq P_{i,j} \leq 1$ for $i,j\in\{1,\ldots,c\}$. A simple graph $G=(V,E)$ with $n$ vertices is said to be generated by the SBM with parameters $C$ and $P$, in which case we write $G\sim SBM(n;C,P)$, when for $u \in V_i$ and $v \in V_j$ with $u \neq v$, the probability that $\{u,v\}\in E$ is $P_{i,j}$.

In this manuscript, we address the resolvability of single and multi-community SBMs. Our work is based on a careful and novel examination of adjacency information only. Such an approach guarantees an upper bound for metric dimension and, for many regimes of the SBM, is a good approximation of full distance information.  With one community, the SBM is equivalent to the Erd\"os-Reny\'i random graph model, i.e. $G \sim G_{n,p}$. We determine highly likely resolving sets for $G$ and improve an asymptotic formula for $\beta(G)$ in this case. When $G \sim SBM(n;C,P)$ has multiple communities, we propose several algorithms to determine subsets $R\subset V$ that resolve $G$ with high probability. Estimating $C$ and $P$ from various real-world networks, we find that, while the ICH algorithm discovers smaller resolving sets, our algorithms are substantially faster, making them practical on large networks.

\section{Adjacency Matrix}
\label{sec:modadjmat}

In determining the metric dimension of a graph $G$ the primary object of interest is its distance matrix, and especially off-diagonal entries. These entries can nevertheless be highly dependent on one another. Indeed, for nodes $u\ne v$ in the same connected component, and $w$ a neighbor of $v$, $\Dbf(u,w) = \Dbf(u,v) - 1$ if $w$ is on a shortest path from $u$ to $v$, $\Dbf(u,w) = \Dbf(u,v) + 1$ if $v$ is on a shortest path from $u$ to $w$, and $\Dbf(u,w) = \Dbf(u,v)$ otherwise. 

Note, however, that for $u\ne v$, $\Dbf(u,v)=1$ if and only if $\Abf(u,v)=1$, and $\Dbf(u,v)\ge2$ if and only if $\Abf(u,v)=0$. In particular, when $\diam(G)\le2$, the off-diagonal entries of $\Dbf$ are (up to a relabelling) in a one-to-one correspondence with the off-diagonal entries of $\Abf$. Thus, in graphs with a diameter at most 2, one should be able to determine metric dimension from direct examination of adjacency, as opposed to full distance, information. To carry forward this idea, we extend the notion of metric dimension of a graph to a matrix. 

Following~\cite{tillquist2019low}, we define the metric dimension $\beta(\Mbf)$ of an arbitrary matrix $\Mbf$ as the smallest number of columns required to make every row unique. (If $\Mbf$ has at least two identical rows, $\beta(\Mbf):=+\infty$.) Accordingly, we say that a set $R$ of columns resolves $\Mbf$ if the rows of the latter are unique when only considering the columns in $R$. In particular, $\beta(\Mbf)$ is the size of a minimal resolving set of columns in $\Mbf$. Note that in the context of graphs, the metric dimension of the distance matrix is the same as that of the graph itself.

Our next result applies to arbitrary graphs, directed or not, with or without loops, and of any diameter.

\begin{lemma}
\label{lem:Dmin2}
Let $G=(V,E)$ be a graph and define $\Abf^*:=\Abf+\diag(2)$, where $\diag(2)$ is a diagonal matrix with the same dimensions as $\Abf$ and 2's along its diagonal. If $R$ resolves $\Abf^*$ then it also resolves $\Dbf$, and if $R$ resolves $\Abf$ then it resolves $\Abf^*$; in particular, $\beta(\Dbf)\le\beta(\Abf^*)\le\beta(\Abf)$.
\end{lemma}

\begin{proof}
Consider $u,v\in V$ such that $u\ne v$. If $R$ resolves $\Abf^*$ then there is $r\in R$ such that $\Abf^*(u,r)\ne\Abf^*(v,r)$. Assume without loss of generality that $u\ne r$; in particular, $\Abf^*(u,r)=\Abf(u,r)$. If $\Abf(u,r)=0$ then $\Dbf(u,r)>1$ and $\Abf^*(v,r)\ge1$. If $\Abf^*(v,r)=1$ then $\Dbf(v,r)=1<\Dbf(u,r)$, whereas if $\Abf^*(v,r)\ge2$ then $\Dbf(v,r)=0<\Dbf(u,r)$. On the other hand, if $\Abf(u,r)=1$ then $\Dbf(u,r)=1$ and $\Abf^*(v,r)=0$ or $\Abf^*(v,r)\ge2$. If $\Abf^*(v,r)=0$ then $\Dbf(v,r)>1=\Dbf(u,r)$, whereas if $\Abf^*(v,r)\ge2$ then $\Dbf(v,r)=0<1=\Dbf(u,r)$. In either case, $\Dbf(v,r)\ne\Dbf(u,r)$, which implies that $R$ resolves $\Dbf$; in particular, $\beta(\Dbf)\le\beta(\Abf^*)$.

Suppose next that $R$ resolves $\Abf$; in particular, if $u\ne v$ then there is $r\in R$ such that $\Abf(u,r)\ne\Abf(v,r)$. Without loss of generality assume that $u\ne r$. If $\Abf(u,r)=1$ and $\Abf(v,r)=0$ then $\Abf^*(u,r)=1$ and $\Abf^*(v,r)\in\{0,2\}$ according to whether $v=r$ or $v\ne r$. Similarly, if $\Abf(u,r)=0$ and $\Abf(v,r)=1$ then $\Abf^*(u,r)=0$ and $\Abf^*(v,r)\in\{1,3\}$. Since in either case $\Abf^*(u,r)\ne\Abf^*(v,r)$, $R$ also resolves $\Abf^*$, implying that $\beta(\Abf^*)\le\beta(\Abf)$.
\end{proof}

When dealing with large graphs, Lemma~\ref{lem:Dmin2} offers a comparatively low complexity alternative to find resolving sets. Indeed, generating the adjacency matrix of a graph requires $O(|V|^2)$ time, while determining full distance information generally assumes access to adjacency information and requires between $O(|V||E|)$~\cite{thorup1999undirected} and $O(|V|^3)$~\cite{floyd1962algorithm} time depending on the nature of the graph and its edge weights.

The next result is tailored for our discussion about the SBM ensemble in the following section.

\begin{corollary}
If $G=(V,E)$ is a graph without loops and with $\diam(G)\le2$ then $\beta(\Dbf)=\beta(\Abf^*)$.
\label{cor:DA*}
\end{corollary}

\begin{proof}
From the previous lemma we know that $\beta(\Dbf)\le\beta(\Abf^*)$. 

Next, suppose that $R\subset V$ resolves $\Dbf$. Because of the hypothesis on $G$, for all $u,v\in V$ with $u\ne v$, $\Abf(u,u)=\Dbf(u,u)=0$ and $\Abf(u,v)=2-\Dbf(u,v)$. As a result, for all $u,v\in V$, $\Abf^*(u,v)=2-\Dbf(u,v)$. In particular, if $r\in R$ is such that $\Dbf(u,r)\ne\Dbf(v,r)$ then $\Abf^*(u,r)\ne\Abf^*(v,r)$. This shows that any resolving set of $\Dbf$ also resolves $\Abf^*$, hence $\beta(\Abf^*)\le\beta(\Dbf)$, which implies the corollary.
\end{proof}

\section{Stochastic Block Model Bounds}
\label{sec:bounds}

In the context of the SBM ensemble, working with $\Abf$ or $\Abf^*$ instead of $\Dbf$ to bound metric dimension has two major advantages. First, unlike $\Dbf$, the entries of $\Abf$ and $\Abf^*$ are independent (though not necessarily identically distributed). Second, due to Lemma~\ref{lem:Dmin2}, the simpler structure of $\Abf^*$ facilitates the discovery of resolving sets of $G$. Of course, $\beta(\Abf^*)$ may be a loose upper-bound of $\beta(\Dbf)$. We show next, however, that across many regimes of the SBM, $\diam(G) \leq 2$ with high probability; in particular, due to Corollary~\ref{cor:DA*}: $\beta(\Dbf)=\beta(\Abf^*)$.

To facilitate the study of the SBM as graph size increases, we define $G_n\sim SBM(n;C_n,P_n)$ in a way analogous to $G \sim SBM(n;C,P)$. However, now $C_n=\{V_{1,n},\ldots,V_{c,n}\}$ is a partition of $\{1,\ldots,n\}$ of some fixed size $c\ge1$, and $P_n$ is a symmetric matrix of dimension $(c\times c)$ with entries in $[0,1]$ for each $n\geq1$. In what follows, $n_i:=|V_{i,n}|$ and $n:=\sum_{i=1}^cn_i$.

\begin{lemma}
Let $s_n(i,j)$ be the number of sets $\{u,v\}$, with $u\ne v$, such that $u\in V_{i,n}$ and $v\in V_{j,n}$, and let
\[\K:=\left\{(i,j)\in\{1,\ldots,c\}^2\hbox{ such that }\limsup_{n\to\infty}s_n(i,j)(1-P_n(i,j))>0\right\}.\]

\noindent If $G_n\sim\SBM(n;C_n,P_n)$, $n_in_j\to\infty$ for each $(i,j)\in\K$, and for all $n$ large enough:
\begin{equation}
\label{eq:P_cond}
\sum_{k=1}^c n_k P_n(i,k)P_n(k,j) \geq C \cdot \ln(n_i n_j),\hbox{ for all $(i,j) \in\K$}
\end{equation}

\noindent for some constant $C>1$, then $\diam(G_n)\le2$ with high probability as $n \rightarrow \infty$.
\label{lem:sbm_diam_const}
\end{lemma}

\begin{proof}
Notice that $s_n(i,i)={n_i\choose 2}$ and, for $i\ne j$, $s_n(i,j)=n_in_j$. Let $W_{i,j}$ denote the number of sets of the form $\{u,v\}$ with $u \in V_{i,n}$ and $v \in V_{j,n}$ such that $d(u,v)>2$, and define $W := \sum_{1 \leq i \leq j \leq c} W_{i,j}$. Due to Markov's inequality:
\[\Prob(\diam(G_n)>2) = \Prob(W > 0) \leq \E(W) = \sum_{1 \leq i \leq j \leq c} \E(W_{i,j}),\]
where
\[\E(W_{i,j}) = s_n(i,j) (1-P_n(i,j)) \prod_{k=1}^c (1-P_n(i,k)P_n(k,j))^{n_k- \llbracket i=k \rrbracket - \llbracket j=k \rrbracket}.\]
If $(i,j)\notin\K$ then $\E(W_{i,j})=o(1)$ since $\E(W_{i,j})\le s_n(i,j) (1-P_n(i,j))$. Instead, if $(i,j)\in\K$ then, using the well-known inequality $(1-x)\le e^{-x}$, %for $x\ge0$,
we find from equation~(\ref{eq:P_cond}) that
\[\E(W_{i,j})
= O\left(n_in_j e^{-\sum\limits_{k=1}^c n_k P_n(i,k)P_n(k,j)}\right)
= O\left((n_in_j)^{1-C}\right)
=o(1).\]
As a result, $\E(W) = o(1)$ and $\Prob(\diam(G_n)>2)\to0$ as $n\to\infty$.
\end{proof}

Assuming that all communities are of the same order of magnitude and applying a stronger, reversed version of condition (\ref{eq:P_cond}), a regime in which $\diam(G)>2$ with high probability may be characterized as well.

\begin{lemma}
Let $G_n \sim SBM(n;C_n,V_n)$ with $n_j = \Theta(n)$ for all $1 \leq j \leq c$. If there exists $1 \leq i,j \leq c$ such that
\begin{small}
\begin{align}
    &\sum_{k=1}^c n_k P_n(i,k)^2 \leq C \cdot \ln(n^2); \text{ or}  && \label{eq:P_cond_conv_a}\\
    &\sum_{k=1}^c n_k P_n(i,k)P_n(k,j) \leq C \cdot \ln(n^2) + \ln(1-P_n(i,j)),\text{ and }\max_{1 \leq k \leq c}P_n(k,j)\leq \frac{1}{2} && \label{eq:P_cond_conv_b}
\end{align}
\end{small}
\noindent for some constant $C<1$, then $\diam(G) > 2$ with high probability as $n \rightarrow \infty$.
\label{lem:sbm_diam_const2}
\end{lemma}

\begin{proof}
Define $W$ and $W_{i,j}$ as in the proof of Lemma~\ref{lem:sbm_diam_const} and let $X_{\{u,v\}}$ be the indicator random variable associated with the event that $d(u,v)>2$ for $u,v \in V_n$. As a result
\[W = \sum_{1 \leq i \leq j \leq c} W_{i,j} = \sum_{u,v \in V_n} X_{\{u,v\}}.\] 

Let $C<1$ and suppose that condition (\ref{eq:P_cond_conv_a}) holds. In particular, since $n_j = \Theta(n)$ for all $1 \leq j \leq c$, there is $i$ such that $P_n(i,j)^2 \leq 2C\ln(n)/n$, which implies that $P_n(i,j) = o(1)$, $nP_n(i,j)^3 = o(1)$, and $nP_n(i,j)^4 = o(1)$. 

Observe that
\begin{align*}
\E(W_{i,i}) 
&= \binom{n_i}{2} (1-P_n(i,i)) \prod_{k=1}^c (1-P_n(i,k)^2)^{n_k - 2\llbracket i=k \rrbracket}  \\
&\sim \frac{n_i^2}{2} e^{-\sum_{k=1}^c n_k P_n(i,k)^2} \\
&\geq \frac{n_i^2}{2} e^{-C \ln(n^2)}  \\
&\rightarrow \infty.
\end{align*}

On the other hand
\begin{align*}
\E(W_{i,i}^2) &= \left(\sum_{|\{u,v\}\cap\{w,t\}|=2}+\sum_{|\{u,v\}\cap\{w,t\}|=1}+\sum_{|\{u,v\}\cap\{w,t\}|=0}\right)\E(X_{\{u,v\}}X_{\{w,t\}}) \\ \\
&= I_2+I_1+I_0,
\end{align*}
where $I_j$ is the summation associated with the indices such that $j=|\{u,v\}\cap\{w,t\}|$. Then $I_2 = \E(W_{i,i}) = o(\E(W_{i,i})^2)$ because $\E(W_{i,i}) \rightarrow \infty$. Moreover:
\begin{align*}
I_1 
&= \binom{n_i}{2} (n_i-2) \big(1-P_n(i,i))^2 \prod_{k=1}^c (1-P_n(i,k)+P_n(i,k)(1-P_n(i,k))^2\big)^{n_k - 3 \llbracket i=k \rrbracket} \\
&\sim \frac{n_i^3}{2} e^{-\sum_{k=1}^c 2n_k P_n(i,k)^2} \\
&= o(\E(W_{i,i})^2),
\end{align*}
and
\begin{align*}
I_0
&= \binom{n_i}{2} \binom{n_i-2}{2} (1-P_n(i,i))^2 \prod_{k=1}^c (1-P_n(i,k)^2)^{2(n_k-4 \llbracket i=k \rrbracket)} \\
&\sim \frac{n_i^4}{4} e^{-\sum_{k=1}^c 2n_kP_n(i,k)^2}  \\
&\sim \E(W_{i,i})^2.
\end{align*}
Hence, $\E(W_{i,i}^2) = I_2+I_1+I_0 \sim \E(W_{i,i})^2$, and the second moment method implies that $\Prob(W_{i,i} >0)\rightarrow 1$. Since $\Prob(W>0)\ge\Prob(W_{i,i} >0)$, it follows that $\diam(G)>2$ with high probability as $n\to\infty$.

Next, suppose that condition (\ref{eq:P_cond_conv_b}) holds for some $C<1$. Then there is $1\leq i,j\leq c$ such that $P_n(i,k)P_n(k,j)\leq2C\ln(n)/n$ for all $1\le k\le c$; in particular, $P_n(i,k)P_n(k,j) =o(1)$ and $nP_n(i,k)^2P_n(k,j)^2=o(1)$. Consequently:
\begin{align*}
\E(W_{i,j})
&= n_in_j (1-P_n(i,j)) \prod_{k=1}^c (1-P_n(i,k)P_n(k,j))^{n_k - \llbracket i=k \rrbracket - \llbracket j=k \rrbracket}  \\
&\sim n_in_j e^{\ln(1-P_n(i,j)) + \sum_{k=1}^c n_k \ln(1-P_n(i,k)P_n(k,j))}  \\
&\sim n_in_j e^{\ln(1-P_n(i,j)) - \sum_{k=1}^c n_k P_n(i,k)P_n(k,j)} \\
&\geq \Theta(1)\cdot n^{2(1-C)} \rightarrow \infty
\end{align*}

On the other hand, taking an approach similar to the one used for decomposing $W_{i,i}^2$, we find that 
\[\E(W_{i,i}^2) = \sum_{(u,v),(w,t)}\E(X_{(u,v)}X_{(w,t)}) = I_2+I_1+I_0\]

\noindent where $u,w \in V_{i,n}$ and $v,t \in V_{j,n}$. Note that, because $i$ and $j$ may be different communities, the order of vertex pairs matters in this case. Here $I_0$ is associated with indices for which $u\ne w$ and $v\ne t$; $I_1$ with indices for which $u=w$ and $v\ne t$, or $u\ne w$ and $v=t$; and $I_2$ with indices for which $u=w$ and $v=t$ so that $I_2 = \E(W_{i,j}) = o(\E(W_{i,j})^2)$. Splitting $I_1$ so that $I_1 = I_1^{(1)} + I_1^{(2)}$, where $I_1^{(1)}$ restricts pairs of indices $(u,v)$ and $(w,t)$ to $u=w$, whereas $I_1^{(2)}$ to $v=t$, we have
\begin{small}
\begin{align*}
I_1^{(1)}
&= n_i \binom{n_j}{2} (1-P_n(i,j))^2 \prod_{k=1}^c \big(1-P_n(i,k)+P_n(i,k)(1-P_n(k,j))^2\big)^{n_k - \llbracket i=k \rrbracket - 2\llbracket j=k \rrbracket}  \\
&\sim \frac{n_in_j^2}{2} e^{2\ln(1-P_n(i,j)) - \sum_{k=1}^c n_k P_n(i,k)P_n(k,j)(2-P_n(k,j))}.
\end{align*}
As a result, using the additional hypothesis that $\max_{1\le k\le c}P_n(k,j)\le1/2$, in particular, $(2-P_n(k,j))\ge3/2$ for all $1\le k\le c$, we obtain that
\begin{align*}
\frac{I_1^{(1)}}{\E(W_{i,j})^2}
&\sim \frac{e^{-\sum_{k=1}^c n_k P_n(i,k)P_n(k,j)(2-P_n(k,j))}}{2n e^{-2\sum_{k=1}^c n_k P_n(i,k)P_n(k,j)}} && \\
&\leq \frac{e^{\frac{1}{2}\sum_{k=1}^c n_k P_n(i,k)P_n(k,j)}}{2n} &&\\
&=O(n^{C-1}) && \\
&= o(1).
\end{align*}
\end{small}

\noindent Similarly, since
\[I_1^{(2)}\sim \frac{n_i^2n_j}{2} e^{2\ln(1-P_n(i,j)) - \sum_{k=1}^c n_k P_n(i,k)P_n(k,j)(2-P_n(i,k))},\]
we also have that
\[\frac{I_1^{(2)}}{\E(W_{i,j})^2}=o(1).\]
As a result, $I_1=o(\E(W_{i,j})^2)$.

Finally,
\begin{small}
\[I_0 = n_in_j(n_i-1)(n_j-1) (1-P_n(i,j))^2 \prod_{k=1}^c (1-P_n(i,k)P_n(k,j))^{2(n_k - 2\llbracket i=k \rrbracket - 2\llbracket j=k \rrbracket)} \cdot Q_{i,j},\]
\end{small}

\noindent where $Q_{i,j}$ is the probability that $\{u,t\}\not\in E$ or $\{u,w\}\notin E$, and $\{v,w\}\not\in E$ or $\{v,t\}\notin E$, and $\{w,u\}\not\in E$ or $\{w,v\}\notin E$, and $\{t,u\}\not\in E$ or $\{t,v\}\notin E$. Since $P_n(i,i)P_n(i,j)\to0$, the probability that $\{u,t\}\not\in E$ or $\{u,w\}\notin E$ (i.e. $1-P_n(i,i)P_n(i,j)$) tends to 1. Similarly, the probability that $\{w,u\}\not\in E$ or $\{w,v\}\notin E$ also tends to 1. Likewise, because  $P_n(i,j)P_n(j,j)\to0$, the probability of $\{v,w\}\not\in E$ or $\{v,t\}\notin E$, and of $\{t,u\}\not\in E$ or $\{t,v\}\notin E$ converge to 1. So $Q_{i,j}\to 1$, and as a result:
\begin{align*}
I_0
&\sim n_i^2n_j^2 (1-P_n(i,j))^2 \prod_{k=1}^c (1-P_n(i,k)P_n(k,j))^{2n_k} \\
&\sim n_i^2n_j^2 e^{2\ln(1-P_n(i,j))-\sum_{k=1}^c 2n_k P_n(i,k)P_n(k,j)} \\
&\sim \E(W_{i,j})^2.
\end{align*}
Then $\E(W_{i,j}^2) = I_2+I_1+I_0 \sim \E(W_{i,j})^2$, and again by the second moment method we have that $\Prob(W>0)\ge \Prob(W_{i,j} >0 )\rightarrow 1$, which completes the proof of the lemma. 
\end{proof}

For $G \sim SBM(n;C,P)$ with $C$ and $P$ constant, Lemma~\ref{lem:sbm_diam_const} implies $\diam(G) \leq 2$ with high probability (see Figure~\ref{fig:sbm_example}). 

There are a variety of other regimes for which graphs in the SBM ensemble have diameter at most 2 with high probability, sometimes unexpectedly. For instance, consider the following two community SBM:
\[P_n = \begin{bmatrix}
o(1) & 1 \\
1 & o(1)
\end{bmatrix}.
\]
In this case, the diameter of each individual community may be quite large, potentially on the order of $(1+o(1))\frac{\ln(n_i)}{\ln(n_iP_n(i,i))}$ if $\ln(n_i) > n_iP_n(i,i)$ and $n_iP_n(i,i)\to\infty$~\cite{chung2001diameter}. Nevertheless, $\diam(G)\le2$ with high probability as $n$ tends to infinity according to Lemma~\ref{lem:sbm_diam_const}. The intuitive reason for this is that the shortest path between two nodes in the same community is likely to include nodes of other communities.

\begin{figure}[h] 
\centering 
\includegraphics[scale=0.5]{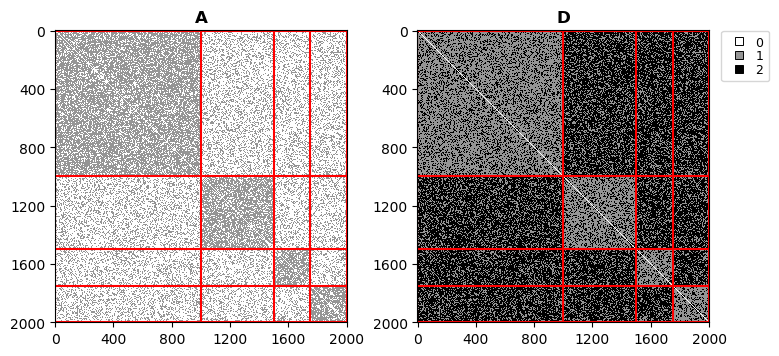} 
\caption[Lof entry]{$G \sim SBM(n;C,P)$ with $n=2000$, community sizes $1000$, $500$, $250$, and $250$, respectively, intra-community adjacency probability $0.7$, and inter-community adjacency probability $0.3$. As expected from Lemma~\ref{lem:sbm_diam_const}, $\diam(G) \leq 2$ in this case.} 
\label{fig:sbm_example} 
\end{figure}

\section{Resolvability of Single Community Random Networks}

The single community SBM is equivalent to the Erd\"os-R\'enyi model $G_{n,p}$. Graphs in this ensemble have $n$ vertices and pairs of distinct vertices are neighbors with probability $p$, independently of all other pairs. We denote the distance and adjacency matrices of $G_{n,p}$ as $\Dbf_{n,p}$ and $\Abf_{n,p}$, respectively, and define $\Abf^*_{n,p} := \Abf + \diag(2)$ as in Lemma~\ref{lem:Dmin2}.

Tight bounds on the metric dimension of $G_{n,p}$ have been established~\cite{bollobas2012metric}. These results focus on several regimes defined by $(n-1)p$, the expected degree of nodes in $G_{n,p}$, and utilize complete distance information whenever possible. In this section, we characterize the metric dimension of $\Abf^*_{n,p}$, and relate it to $\Dbf_{n,p}$.

\begin{lemma}
Let $q=\min\{p,1-p\}$. If $\frac{\ln(n)}{n}=o(q)$ then $\beta(\Abf^*_{n,p}) \leq \lceil\frac{-2\ln(n)}{\ln(p^2+(1-p)^2)}\rceil$ with high probability. On the other hand, if $\frac{\ln^2(n)}{n^{1-C/2}} = o(q)$ for some constant $1<C<2$ then $\beta(\Abf^*_{n,p})\geq \lfloor\frac{-C\ln(n)}{\ln(p^2+(1-p)^2)}\rfloor$ with high probability. 
\label{lem:betaAnp}
\end{lemma}

\emph{Proof.} In what follows, $r_i:=p^i+(1-p)^i$ for $i\ge1$.

To show the first part of the lemma set $k:= \lceil\frac{-2\ln(n)}{\ln(r_2)}\rceil$; in particular, $n^2r_2^k\le1$. We claim that under the hypothesis on $p$, $k=o(n)$. In fact, since this is certainly the case when $p$ is bounded away from $0$ and $1$, and $k$ is a symmetric function of $p$ about $p=1/2$, it suffices to verify the claim only when $p\to0^+$. But, in this case, $\ln(r_2)\sim -2p$, hence $k\sim\frac{\ln(n)}{p}=o(n)$ as claimed.

Next, select $k$ columns in $\Abf^*_{n,p}$, and let $W$ denote the number of pairs of rows that are identical over the selected columns. Due to the first moment method, and since rows corresponding to selected columns are guaranteed to be unique as a result of diagonal entries, we obtain:
\begin{equation*}
\Prob(W>0)\leq \E(W) = \binom{n-k}{2}r_2^k 
\end{equation*}
In particular, since $\binom{n-k}{2} \leq \binom{n}{2}$, we find that:
\[\Prob(W>0)\leq \binom{n}{2}r_2^k\leq \frac{n^2r_2^k}{2} \leq \frac{1}{2}.\]

Define $W_i$, for $1\le i\le\lfloor n/k\rfloor$, as the number of pairs of rows in $\Abf^*_{n,p}$ that are identical when the latter is restricted to the columns with labels $(i-1)k+1,\ldots,ik$. Since $W_1,\ldots,W_{\lfloor n/k\rfloor}$ are independent random variables and $k=o(n)$, the above inequality implies that
\[\Prob(\beta(\Abf^*_{n,p}) > k)\leq \Prob\left(W_i>0,\hbox{ for all }1\le i\le\left\lfloor \frac{n}{k}\right\rfloor\right)\le\left(\frac{1}{2}\right)^{\lfloor \frac{n}{k}\rfloor}=o(1).\]
This shows the first part of the lemma.

To show the second part, let $1\le k\le n$ be an integer, and $S$ denote the total number of resolving subsets of $k$ columns in $\Abf^*_{n,p}$. In particular, 
\begin{equation}
\Prob(S>0) \leq \E(S) = \binom{n}{k} \cdot\Prob(W=0),
\label{ide:key1}
\end{equation}
with $W$ defined in an analogous manner as above. To bound the last factor we use an exponential bound from~\cite{janson1998new}. For this, let $\mathcal{I}$ be the set with elements of the form $\{u,v\}$, where $u$ and $v$ are different rows in $\Abf^*_{n,p}$, and let $X_{\{u,v\}}$ be the indicator function of the event that rows $u$ and $v$ are indistinguishable with respect to the given set of $k$ columns. The dependency graph associated with $\mathcal{I}$ is $\Gamma=(\mathcal{I}, E_{\Gamma})$ where $s,t \in \mathcal{I}$ are adjacent when $|s \cap t| > 0$. Following ~\cite[Theorem 3]{janson1998new} we have that
\begin{equation}
\Prob(W=0) \leq 
e^{-\alpha},
\label{ide:key2}
\end{equation}
where
\begin{align*}
\alpha &:= \min\left\{\frac{\mu^2}{8\Delta}, \frac{\mu}{6\delta}, \frac{\mu}{2}\right\}\\
\mu &:= \E(W) = \binom{n-k}{2}r_2^k \\ 
\Delta &:= \sum_{\{s,t\} \in E_{\Gamma}} \E(X_sX_t) \leq 3\binom{n}{3}r_3^k\\
\delta &:= \max_{s \in \mathcal{I}} \sum_{\substack{t \in \mathcal{I}, \{s,t\} \in E_{\Gamma}}} r_2^k \leq 2(n-2)r_2^k.
\end{align*}

Next, assume that $\frac{\ln^2(n)}{n^{1-C/2}}=o(q)$ for some constant $1<C<2$. It follows from the inequalities~(\ref{ide:key1}) and~(\ref{ide:key2}) that 
\begin{equation}
P(S>0)\leq e^{k\ln(n)-\alpha} = n^{k\big(1-\frac{\alpha}{k\ln(n)}\big)}.
\label{ine:master}
\end{equation}

Suppose first $p\to0^+$; in particular, $\ln(r_2)\sim -2p$.

Due to the assumptions on $C$ and $q$:
\begin{equation}
\frac{\ln^2(n)}{n}\ll\frac{\ln^2(n)}{n^{2-C}}\ll \frac{\ln^2(n)}{n^{1-C/2}}\ll p.
\label{ide:orders}
\end{equation}
In particular, $\ln(n)/n\ll p$. So if we define $k=\lfloor\frac{-C\ln(n)}{\ln(r_2)}\rfloor$ then $k\sim\frac{C\ln(n)}{2p}\ll n$ (equation~(\ref{ine:master}) requires $1\le k\le n$). Since $\binom{n-k}{2} \sim \binom{n}{2}$, this implies that
\[\mu\sim\frac{n^2r_2^k}{2}.\]
We show now that the right-hand side on equation (\ref{ine:master}) converges to $0$. For this, it suffices to prove that $\lim_{n\to\infty}\frac{\alpha}{k\ln(n)}=+\infty$, which we show considering each possibility for the minimum in $\alpha$.

If $\alpha = \frac{\mu^2}{8\Delta}$ and since $r_2^k\ge n^{-C}$, $\ln(r_3)/\ln(r_2)\ge3/2$ for all $0 < p < 1$, and $r_3\ge1/4$, then
\begin{align*}
\frac{\alpha}{k\ln(n)}
&\sim\frac{nr_2^{2k}p}{8C\ln^2(n)r_3^k}\\
&\ge\frac{n^{1-2C}p}{8C\ln^2(n)r_3^k}\\
&\ge\frac{r_3n^{1-2C}e^{C\ln(n)\cdot\ln(r_3)/\ln(r_2)}p}{8C\ln^2(n)}\\
&\ge\frac{n^{1-C/2}p}{32C\ln^2(n)}\longrightarrow\infty,
\end{align*}
because of equation~(\ref{ide:orders}).

If $\alpha = \frac{\mu}{6\delta}$, then
\begin{align*}
\frac{\alpha}{k\ln(n)} &\sim \frac{n^2r_2^kp}{6nr_2^kC\ln^2(n)} \\
                       &= \frac{np}{6C\ln^2(n)} \longrightarrow \infty
\end{align*}

Finally, if $\alpha = \frac{\mu}{2}$, then
\begin{align*}
\frac{\alpha}{k\ln(n)} &\sim \frac{n^2r_2^kp}{C\ln^2(n)} \\
                       &\geq \frac{n^{2-C}p}{C\ln^2(n)} \longrightarrow \infty
\end{align*}

Next, suppose that $p\in(0,1)$; in particular, $p$ is bounded away from 0 and 1 and $\ln(r) = \Theta(-2p)$. Following the same argument as above, it can be shown that $\lim_{n \rightarrow \infty} n^{k\big(1-\frac{\alpha}{k\ln(n)}\big)} = 0$. Hence, in all cases $P(S>0) \leq \binom{n}{k}P(W=0) \rightarrow 0$ and the conclusion follows. \hfill $\square$ \\*

The following is now a direct consequence of Lemma~\ref{lem:betaAnp}.

\begin{corollary}
Let $q=\min\{p,1-p\}$. If $\frac{\ln^2(n)}{n^{1-C/2}}=o(q)$ for $1<C<2$, then $\beta(\Abf^*_{n,p})\sim\frac{-2\ln(n)}{\ln(p^2+(1-p)^2)}$, with respect to convergence in probability.
\label{cor:betaAprobconv}
\end{corollary}

We are now able to relate the metric dimensions of $\Abf^*_{n,p}$ and $\Dbf_{n,p}$ in the Erd\"os-R\'enyi model. We note that while the rate of convergence of $p$ towards zero implied by our methods is weaker than the one found in~\cite{bollobas2012metric}, our next result improves on the rate of convergence towards one.

\begin{corollary}
Let $q=\min\{p,1-p\}$. For $\frac{\ln^2(n)}{n^{1-C/2}} =  o(q)$ where $1<C<2$, $\beta(G_{n,p})=\beta(\Abf^*_{n,p})\sim\frac{-2\ln(n)}{\ln(p^2+(1-p)^2)}$ with high probability as $n$ tends to infinity.
\label{cor:betaGnp}
\end{corollary}

\begin{proof}
Since $\sqrt{\frac{2\ln(n)}{n}} \ll p$, $\diam(G_{n,p}) \leq 2$ with high probability~\cite[Theorem 8.5]{blum2016foundations}. Then Corollary~\ref{cor:DA*} implies that $\beta(G_{n,p})=\beta(\Abf^*_{n,p})$. The result follows from Corollary~\ref{cor:betaAprobconv}.
\end{proof}

For the particular case with $p=1/2$, it follows from~\cite{babai1980random} that $\beta(G_{n,1/2})\leq \lceil 3\ln(n) / \ln(2) \rceil$ with high probability as $n$ goes to infinity. Indeed, an $O(n^2)$ algorithm for checking whether or not two random graphs of size $n$ are isomorphic presented in~\cite{babai1980random} uses the $\lceil 3\ln(n) / \ln(2) \rceil$ vertices of highest degree as a kind of resolving set. Intuitively, however, since the degree distribution of $G_{n,1/2}$ is Binomial and hence concentrated around its mean~\cite{arratia1990erdos,cramer1938nouveau}, and since $G_{n,1/2}$ should be a highly homogeneous graph, the degree differences between nodes should be inconsequential when searching for a resolving set. Our next result shows that this is precisely the case and generalizes the bound for $\beta(G_{n,1/2})$ to arbitrary Erd\"os-R\'enyi graphs.

\begin{corollary}
\label{cor:er_bound}
Any set of $k\geq\frac{-3 \ln(n)}{\ln(p^2+(1-p)^2)}$ nodes resolves $G_{n,p}$ with probability greater than or equal to $(1-1/2n)$. In particular: \[\beta(G_{n,p})\le\left\lceil\frac{-3 \ln(n)}{\ln(p^2+(1-p)^2)}\right\rceil,\] with high probability as $n$ increases.
\end{corollary}

\begin{proof}
Let $r:=p^2+(1-p)^2$. Due to Lemma~\ref{lem:Dmin2}, it suffices to show that any set of $k\geq-3 \ln(n)/\ln(r)$ columns resolves the matrix $\Abf^*_{n,p}$ for $G_{n,p}$ with probability at least $(1-1/2n)$. (The rate of convergence towards $1$ of this probability is somewhat arbitrary, and was selected only to reproduce the result in~\cite{babai1980random} for $p=1/2$.)

Following an analogous argument as in the proof of Lemma~\ref{lem:betaAnp}, let $W$ denote the number of pairs of rows that cannot be distinguished over the selected columns. Due to the first moment method, we have
\begin{equation}
\Prob(W>0)\leq \E(W) \leq \binom{n}{2} r^k \leq \frac{n^2r^k}{2} \le\frac{1}{2n},
\end{equation} 
where the last inequality assumes that $k\geq-3 \ln(n)/\ln(r)$. This shows the corollary.
\end{proof}

\section{Resolvability of Random Networks with Multiple Communities}

In this section, our goal is to resolve---with as few nodes as possible---$G \sim SBM(n;C,P)$ with known parameters; in particular, our analysis is no longer asymptotic.

More precisely, for a given user-defined threshold $0<\alpha<1$, we aim to prescribe $k_1,\ldots,k_c$, such that $k_i$ nodes selected uniformly at random from community $i$ resolve the multiple community graph with probability at least $(1-\alpha)$. The challenge is to meet this threshold minimizing
\[k:=\sum_{i=1}^c k_i.\] 
We emphasize that nodes in the alleged resolving set are selected at random (within each community) to have a randomized algorithm capable of handling large networks in practice.

In what follows, $G \sim SBM(n;C,P)$ and $\Abf^*$ is the modified version of the adjacency matrix of $G$ with 2's along the diagonal (see Lemma~\ref{lem:Dmin2}). Observe that while the entries in this matrix are independent, rows from different communities are not necessarily identically distributed. Furthermore, define
\[r(i,j,\ell) := P(i,\ell)P(j,\ell)+(1-P(i,\ell))(1-P(j,\ell)),\] 
i.e. the probability that $\Abf^*(u,w)=\Abf^*(v,w)$ for given but distinct nodes $u\in V_i$, $v\in V_j$, and $w \in V_{\ell}$.

In what remains of this section, $W$ denotes the number of pairs of rows in $\Abf^*$ that collide over the $k$ random columns. Then, due the first moment method and recalling that the diagonal entries of $\Abf^*$ guarantee that rows corresponding to selected columns are unique, we have:
\begin{equation}
\label{eq:monotonic_f}
\Prob(W>0)\leq \E(W) \leq f(k_1, \dots, k_c),
\end{equation}
where $f:\mathbb{N}_0^c \rightarrow \mathbb{R}$ is defined as \begin{equation}
f(\kbf):=\sum_{1\leq i \leq j \leq c} s(i,j) \prod_{\ell=1}^c r(i,j,\ell)^{\kbf_{\ell}},
\label{def:f(bfk)}
\end{equation}
with $s(i, j):=\binom{|V_i|}{2}$ when $i=j$, and $s(i, j):=|V_i||V_j|$ when $i \neq j$. We note that $f(\kbf)$ is a tight upper-bound of $\E(W)$ when $\kbf_i\ll|V_i|$ for each community $i$, which is what we usually expect.

We aim to determine $\kbf\in\NN_0^c$ that minimizes $\sum_{i=1}^c\kbf_i$ subject to $f(\kbf)\le\alpha$. A very useful property for this integer programming problem is that $f(\ybf)\leq f(\xbf)$ whenever $\xbf\leq\ybf$, i.e. when each coordinate of $\xbf$ is at least as small as the corresponding coordinate of $\ybf$. In other words, $f$ is decreasing with respect to the partial order $\le$. In particular, if $f(\ybf) > \alpha$ for certain $\ybf \in \mathbb{N}_0^c$ then all $\xbf \leq \ybf$ may be eliminated from the search space.

In what follows, we assume that $f(|V_1|,\ldots,|V_c|)\le\alpha$. This condition is necessary and sufficient for the existence of a feasible point because $f$ is decreasing.

\subsection{The MINE Algorithm} Because the function $f$ is decreasing, our integer programming problem may be alternatively phrased as determining the largest $h\ge1$ such that $f(\kbf')>\alpha$ for each $\kbf'\in S_{h-1}$. To accomplish this in an efficient manner, we require the following definitions.

\begin{definition}
For each integer $h\ge0$, define $S_h := \{\kbf\in \mathbb{N}_0^c\,|\,\sum_{i=1}^c \kbf_i = h\}$. Further, define the \underline{D}ownward operator $D : \cup_{h\ge1}S_h \rightarrow \cup_{h\ge0}S_h$, which subtracts 1 from the leftmost strictly positive entry of $\kbf \in S_h$, and the \underline{U}pward operator $U: \cup_{h\ge0}S_h \rightarrow \cup_{h\ge1}S_h$, which adds 1 to $\kbf_1$.
\end{definition}

The terminology in the above definition is motivated by the following observations: if $\xbf\in S_h$, with $h\ge1$, then $D(\xbf)\in S_{h-1}$, and if $\ybf\in S_h$, with $h\ge0$, then $U(\ybf)\in S_{h+1}$.

\begin{definition}
$\preceq$ denotes the reverse lexicographic order on $\NN_0^c$, i.e. $\xbf \preceq\ybf$ if and only if $\xbf=\ybf$, or $\xbf_i>\ybf_i$ for the smallest index $1 \leq i \leq c$ with $\xbf_i \neq \ybf_i$.
\end{definition}

The key insight behind our method is given by the following result.

\begin{lemma}
\label{lem:sbm_algo_order}
Let $h\ge1$. If $\kbf\in S_h$ and $\kbf'\in S_{h-1}$ are such that $\kbf'\preceq D(\kbf)$, then $\kbf'< U(\kbf')\preceq\kbf$.
\end{lemma}

\begin{proof} 
Define $\mathbf{a}:=D(\kbf)$ and $\mathbf{b}=U(\kbf')$. Since $\kbf'_1<\mathbf{b}_1$ but $\kbf'_i=\mathbf{b}_i$ for $1<i\le c$, we have $\kbf'<\mathbf{b}$. On the other hand, if $\kbf_1 = 0$ then $\mathbf{b}_1 \geq 1$, and $\mathbf{b} \preceq \kbf$ as claimed. Instead, if $\kbf_1 > 0$ then $\kbf = U(\mathbf{a})$. Since $\kbf' \preceq \mathbf{a}$, either $\kbf' = \mathbf{a}$ or $\kbf'_i > \mathbf{a}_i$ for the smallest index $1 \leq i \leq c$ with $\kbf'_i \neq \mathbf{a}_i$. Assuming the former, $\mathbf{b}=U(\mathbf{a})=\kbf$. Assuming the latter, $i$ must also be the first index at which $\mathbf{b}_i \neq \kbf_i$ and $\mathbf{b}_i  > \kbf_i $. Thus, in either case, $\mathbf{b}\preceq \kbf$ as claimed.
\end{proof}

Lemma~\ref{lem:sbm_algo_order} allows us to use the Downward operator to reduce $h$ but without having to explore all of $S_{h-1}$ for an optimum of our integer programming problem. Indeed, \textit{suppose that for some $h\ge1$ we have found $\kbf\in S_h$ such that $f(\mathbf{a})>\alpha$, for all $\mathbf{a}\prec\kbf$ with $\mathbf{a}\in S_h$, but $f(\kbf)\le\alpha$.} Then, due to the lemma, if $\kbf'\in S_{h-1}$ and $\kbf'\prec D(\kbf)$ then $\kbf'<U(\kbf')\prec\kbf$; in particular, because $f$ is decreasing, $f(\kbf')\ge f(U(\kbf'))>\alpha$ i.e. $f(\kbf')>\alpha$. In conclusion, \textit{$D(\kbf)$ is the smallest point in $S_{h-1}$ with respect to the total order $\preceq$ that may be feasible.} 

Algorithm~\ref{sbm_algo}, or MINE (Minimizing Indistinguishable Nodes via Expectation), implements the above observation as follows.

At the onset, $\kbf$ is determined such that $f(\kbf)\le\alpha$ and $f(\kbf-\mathbf{e})>\alpha$, for each canonical vector 
$\mathbf{e}\in\mathbb{N}_0^c$ such that $(\kbf-\mathbf{e})\in\mathbb{N}_0^c$. This $\kbf$ is found using a forward greedy search (Algorithm~\ref{greedy_min_algo})---see Figure~\ref{fig:algo_ex} top left, followed by a backward greedy search (Algorithm~\ref{greedy_max_algo})---see Figure~\ref{fig:algo_ex} top right. We emphasize that the second requirement on $\kbf$ is not essential; however, it works well in practice to start the optimum search with a smaller value of $h$. 

Set $h\leftarrow\sum_{i=1}^c\kbf_i$, with $\kbf$ as above. MINE proceeds searching for feasible points in $S_{h-1}$ starting 
with $(h-1,0,\ldots,0)\in\NN_0^c$, the smallest point in $S_{h-1}$ w.r.t. the order $\preceq$. Each time a point in $S_{h-1}$ is deemed unfeasible, its successor w.r.t. $\preceq$ in $S_{h-1}$ is determined (Algorithm~\ref{next_point}), which is checked for feasibility---see Figure~\ref{fig:algo_ex} bottom left. This process continues until either, no point in $S_{h-1}$ is deemed feasible, in which case $\kbf$ is an optimum, or a point $\xbf\in S_{h-1}$ is found such that $f(\xbf)\le\alpha$, in which case we reset $h\leftarrow(h-1)$, and the search for an optimum restarts at $\kbf\leftarrow D(\xbf)$ and follows its successors in $S_{h-1}$ (always w.r.t. $\preceq$). MINE repeats the last process, incrementally reducing the value of $h$, until a successor of the last feasible $\kbf$ has no successor in $S_h$, in which case $\kbf$ is an optimum---see Figure~\ref{fig:algo_ex} bottom right.

\begin{figure}[h!] 
\centering 
\includegraphics[scale=0.65]{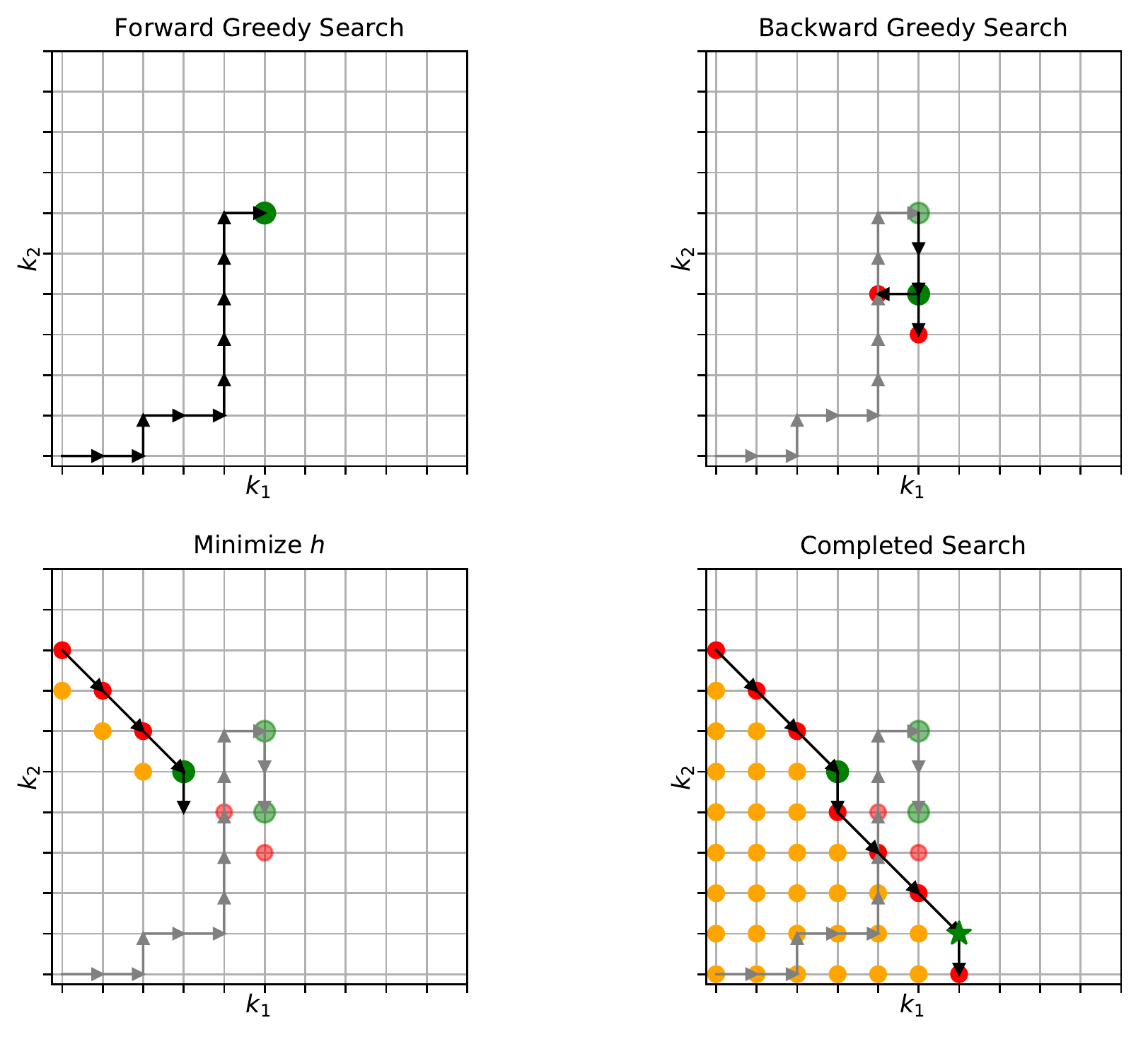} 
\caption[Lof entry]{A schematic of the main steps of the MINE algorithm with two communities. $f(\kbf) \leq \alpha$ at green points, $f(\kbf) > \alpha$ is verified explicitly at red points, $f(\kbf) > \alpha$ is implicit at orange points, and the green star represents an optimal solution $\kbf$. Black arrows record the most recent stage of the algorithm while grey arrows record older stages.} 
\label{fig:algo_ex} 
\end{figure}

\section{Illustration on Derived Synthetic Networks}

To gain an understanding of the MINE algorithm (Algorithm~\ref{sbm_algo}) in terms of both runtime and resolving set size, we applied it to synthetic networks generated via the SBM with parameters estimated from several real world networks---see Table~\ref{tab:network_descriptions} and Figure~\ref{fig:real_world_networks}. 

Community sizes were altered so that graphs of varying sizes could be considered. Specifically, we set
\[|V_i| = n\cdot\frac{|\hat{V_i}|}{\sum_{j=1}^c |\hat V_j|},\]
rounded to the nearest integer, where $\hat{V_i}$ is a labeled community in the original network. We focused on networks with $n\approx10000$ nodes so that the ICH algorithm may be run in a reasonable amount of time. Standardizing vertex labels so that communities consist of consecutively labeled nodes, we generated 30 synthetic networks for each set of parameters. 

A total of five algorithms were considered for each network---see Table~\ref{tab:algo_descriptions}. Techniques depending on selection of random nodes were applied 50 times to each network. Due to long run times, however, the ICH algorithm was only run 10 times. The time required to compute full distance information for networks half this size was already substantial (Table~\ref{tab:dist_mat_times}) and, since the fraction of shortest paths with length greater than 2 is less than $0.001$ for all networks with $5000$ nodes, except those based on the political blogs network (Table~\ref{tab:frac_long_paths}), we expect $\Abf^*=2-\Dbf$ and $\Dbf$ to be nearly identical in most cases for large networks. All results and associated code are available on GitHub (\url{https://github.com/riti4538/SBM-Metric-Dimension}).

\begin{table}
\centering 
\scriptsize
\begin{tabular}{l|ccc}
\toprule
\multicolumn{1}{>{\centering\arraybackslash}m{28mm}}{Network} & \multicolumn{1}{>{\centering\arraybackslash}m{18mm}}{Communities} & \multicolumn{1}{>{\centering\arraybackslash}m{28mm}}{Size (community sizes)} & \multicolumn{1}{>{\centering\arraybackslash}m{30mm}}{$P$} \\
\midrule
Political Blogs~\cite{adamic2005political} & 2 & 1222 $(586, 636)$ & $\begin{bmatrix} 0.043 & 0.004 \\ 0.004 & 0.039 \end{bmatrix}$ \\[5mm]
Political Books~\cite{krebspoliticalbooks} & 3 & 105 $(49, 43, 13)$ & $\begin{bmatrix} 0.162 & 0.006 & 0.053 \\ 0.005 & 0.190 & 0.043 \\ 0.053 & 0.043 & 0.115 \end{bmatrix}$ \\[5mm]
Zachary's Karate Club~\cite{zachary1977information} & 2 & 34 $(17, 17)$ & $\begin{bmatrix} 0.257 & 0.038 \\ 0.038 & 0.228 \end{bmatrix}$ \\[5mm]
\emph{David Copperfield}~\cite{newman2006finding} & 2 & 112 $(58, 54)$ & $\begin{bmatrix} 0.063 & 0.098 \\ 0.098 & 0.010 \end{bmatrix}$ \\[5mm] 
Primary School~\cite{stehle2011high} & 3 & 236 $(110, 112, 14)$ & $\begin{bmatrix} 0.198 & 0.204 & 0.160 \\ 0.204 & 0.268 & 0.166 \\ 0.160 & 0.166 & 0.297 \end{bmatrix}$ \\
\bottomrule
\end{tabular}
\caption{Descriptions of real world networks.}
\label{tab:network_descriptions}
\end{table}

\begin{figure}[h!]
\centering
\begin{subfigure}[t]{0.49\textwidth}
    \centering
    \includegraphics[scale=0.23]{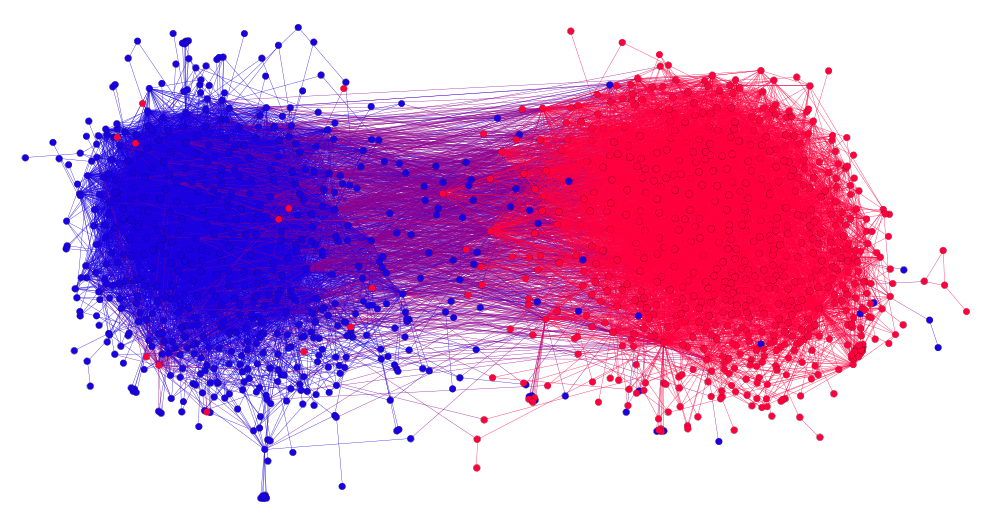}
    \caption{Political blogs}\label{fig:image11}
\end{subfigure}
    \hfill
\begin{subfigure}[t]{0.49\textwidth}
    \centering
    \includegraphics[scale=0.23]{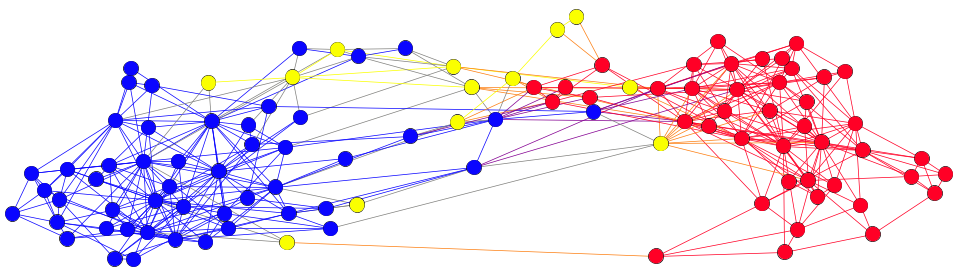}
    \caption{Political books}\label{fig:image12}
\end{subfigure}
\begin{subfigure}[t]{0.49\textwidth}
  \centering
  \includegraphics[scale=0.23]{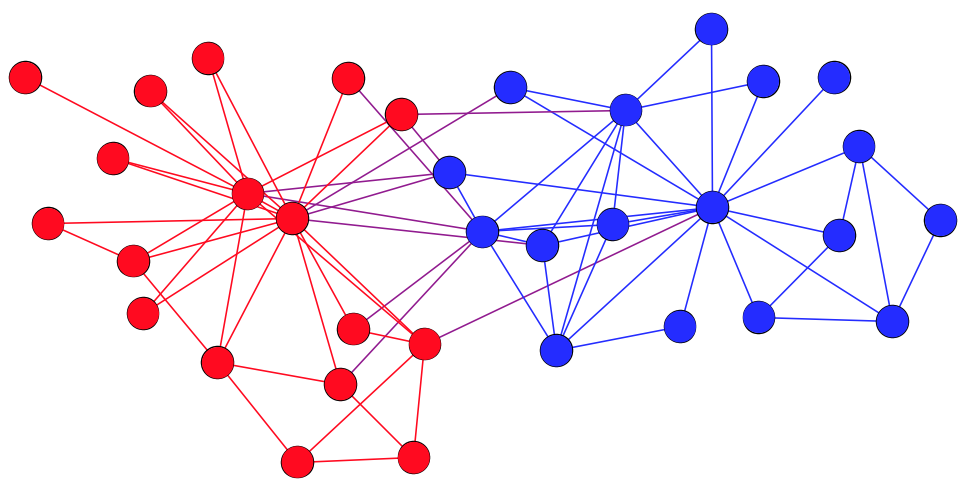}
  \caption{Karate Club}\label{fig:image21}
\end{subfigure}
    \hfill
\begin{subfigure}[t]{0.49\textwidth}
    \centering
    \includegraphics[scale=0.23]{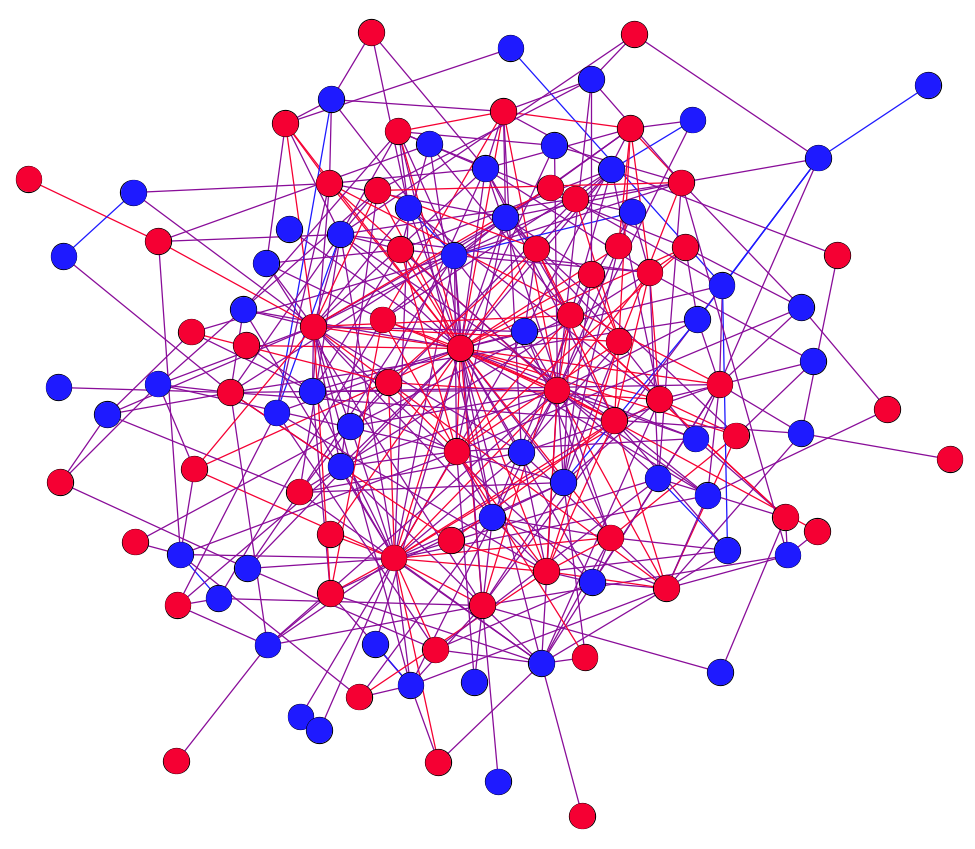}
    \caption{\emph{David Copperfield}}\label{fig:image22}
\end{subfigure}
\begin{subfigure}{0.49\textwidth}
  \centering
  \includegraphics[scale=0.23]{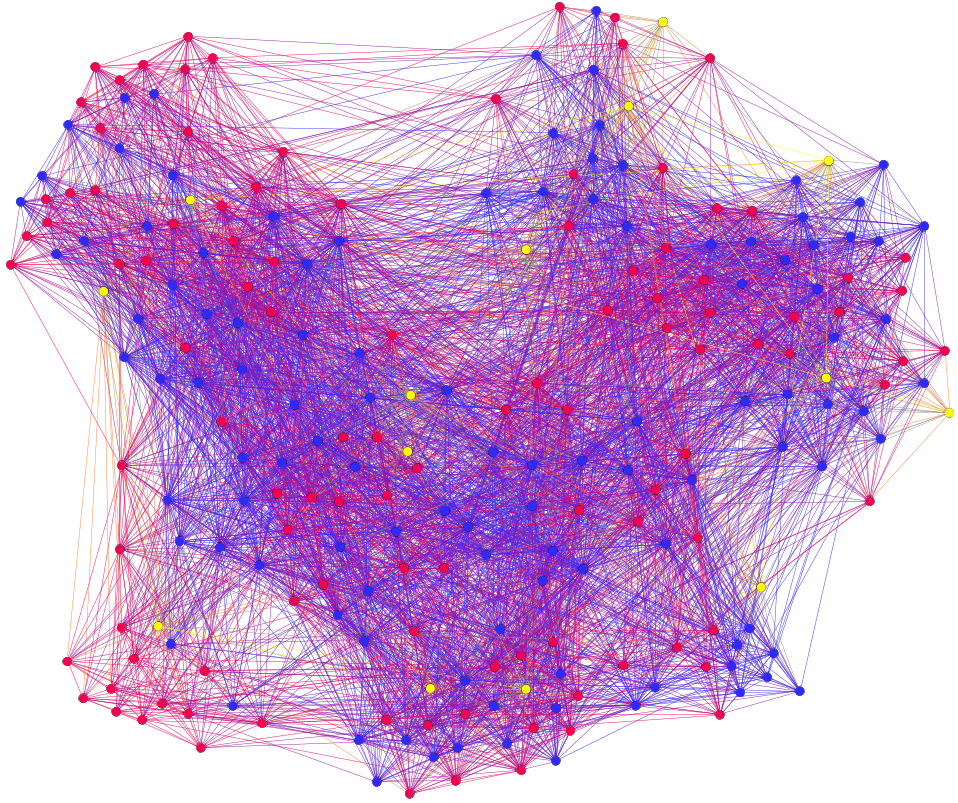}
  \caption{Primary school}\label{fig:image3}
\end{subfigure}
\caption{Visualizations of real world networks with nodes colored according to community membership (see Table~\ref{tab:network_descriptions}) generated with Gephi~\cite{bastian2009gephi}.}
\label{fig:real_world_networks}
\end{figure}

\begin{table}
\centering 
\scriptsize
\begin{tabular}{l|p{5cm}c}
\toprule
\multicolumn{1}{>{\centering\arraybackslash}m{23mm}}{Algorithm} & \multicolumn{1}{>{\centering\arraybackslash}m{50mm}}{Description} & \multicolumn{1}{>{\centering\arraybackslash}m{28mm}}{Randomized} \\
\midrule
MINE & Minimizing Indistinguishable Nodes via Expectation (Algorithm~\ref{sbm_algo}) & No \\[8mm]
ICH & Use the Information Content Heuristic to find a resolving set~\cite{hauptmann2012approximation} & Yes \\[8mm]
Greedy & Iteratively select nodes from communities according to Algorithm~\ref{greedy_min_algo} & Yes \\[8mm]
Preorder & Order nodes within communities based on the function $f(\cdot)$ (equation (\ref{eq:monotonic_f})). Consider communities based on the order prescribed by Algorithm~\ref{greedy_min_algo} & No \\[12mm]
Random & Iteratively select nodes uniformly at random from the full graph & Yes \\
\bottomrule
\end{tabular}
\caption{Descriptions of algorithms.}
\label{tab:algo_descriptions}
\end{table}

\begin{table}[h]
\scriptsize
\centering
\begin{tabular}{cccccc}
\toprule
Network & Distance Matrix Construction Time \\
\midrule
Political Blogs & $31095.5 \pm 212.12$ \\
Political Books & $30851.17 \pm 293.2$ \\
Karate Club & $30134.45 \pm 251.75$ \\
\emph{David Copperfield} & $29950.99 \pm 198.65$ \\
Primary School & $30137.52 \pm 256.05$ \\
\bottomrule
\end{tabular}
\caption[LoF entry]{Time in seconds required to generate full distance matrices for networks of size $n=5000$ using the NetworkX Floyd-Warshall implementation~\cite{hagberg2008exploring}.}
\label{tab:dist_mat_times}
\end{table}

\begin{table}[h]
\centering
\scriptsize
\begin{tabular}{cc}
\toprule
Network & Fraction of Paths with Length $>2$ \\ 
\midrule
Political Blogs & $0.216$ \\ 
Political Books & $0.001$ \\ 
Karate Club & $5.770 \times 10^{-10}$ \\ 
\emph{David Copperfield} & $1.538 \times 10^{-18}$ \\ 
Primary School & $3.095 \times 10^{-12}$ \\ 
\bottomrule
\end{tabular}
\caption[LoF entry]{Fraction of shortest path lengths greater than 2 across all synthetic networks of size $n=5000$ for each set of real world parameters.}
\label{tab:frac_long_paths}
\end{table}

For threshold values $\alpha \in \{0.005, 0.01, 0.1, 0.2\}$, MINE is very fast. Resolving sets for networks with $\sim 10000$ nodes are found in less than 3 seconds (Table~\ref{tab:algo_adj_time}). The difference in run times as $\alpha$ decreases, that is as we become more strict regarding our certainty that a resolving set will be found, is small but generally increasing. The change in resolving set sizes as $\alpha$ decreases is more pronounced (Table~\ref{tab:algo_adj_res_set}). In particular, a larger subset of nodes is required to guarantee resolvability when $\alpha$ is smaller. The difference, however, between the number of nodes required when $\alpha$ is $0.005$ versus $0.2$ is comparatively small, less than $1\%$ the size of the full network. So, while considering a set with a $20\%$ probability of failing to resolve a network seems impractical, the distinction between such a set and a set with only a $0.5\%$ probability of failing to resolve a network translates into a small percentage of the network size.

Overall, the size of resolving sets found with MINE depends substantially on $P$. For instance, networks based on the political blogs data set are relatively sparse and require more nodes to uniquely identify all nodes with high probability.

\begin{table}[h]
\scriptsize
\centering
\begin{tabular}{ccccc}
\toprule
Network & $\alpha = 0.005$ & $\alpha = 0.01$ & $\alpha = 0.1$ & $\alpha = 0.2$\\ 
\midrule
Political Blogs & $2.9 \pm 0.25$ & $2.87 \pm 0.28$ & $2.58 \pm 0.2$ & $2.55 \pm 0.24$ \\ 
Political Books & $2.5 \pm 0.06$ & $2.41 \pm 0.12$ & $2.07 \pm 0.14$ & $1.97 \pm 0.18$ \\ 
Karate Club & $0.66 \pm 0.07$ & $0.64 \pm 0.08$ & $0.6 \pm 0.14$ & $0.59 \pm 0.13$ \\ 
\emph{David Copperfield} & $1.04 \pm 0.06$ & $1.04 \pm 0.08$ & $0.94 \pm 0.05$ & $0.93 \pm 0.08$ \\ 
Primary School & $0.72 \pm 0.04$ & $0.7 \pm 0.1$ & $0.59 \pm 0.05$ & $0.57 \pm 0.08$ \\ 
\bottomrule
\end{tabular}
\caption[LoF entry]{Time in seconds required to determine highly likely resolving sets using MINE ($n=10000$).}
\label{tab:algo_adj_time}
\end{table}

\begin{table}[h]
\scriptsize
\centering
\begin{tabular}{ccccc}
\toprule
Network & $\alpha = 0.005$ & $\alpha = 0.01$ & $\alpha = 0.1$ & $\alpha = 0.2$\\ 
\midrule
Political Blogs & $506.73 \pm 0.51$ & $491.43 \pm 0.62$ & $440.87 \pm 0.43$ & $425.53 \pm 0.5$ \\ 
Political Books & $151.0 \pm 0.0$ & $146.0 \pm 0.0$ & $131.0 \pm 0.0$ & $126.0 \pm 0.0$ \\ 
Karate Club & $84.97 \pm 0.18$ & $82.0 \pm 0.0$ & $73.0 \pm 0.0$ & $71.0 \pm 0.0$ \\ 
\emph{David Copperfield} & $157.03 \pm 0.18$ & $152.33 \pm 0.47$ & $137.0 \pm 0.0$ & $132.0 \pm 0.0$ \\ 
Primary School & $58.0 \pm 0.0$ & $56.0 \pm 0.0$ & $50.0 \pm 0.0$ & $48.0 \pm 0.0$ \\ 
\bottomrule
\end{tabular}
\caption[LoF entry]{Resolving set sizes determined using MINE ($n=10000$).}
\label{tab:algo_adj_res_set}
\end{table}

The ICH algorithm, on the other hand, is extremely slow taking between 2700 and 20000 seconds, depending on the underlying network, to find resolving sets on networks with 10000 nodes (Table~\ref{tab:adj_time}). While full distance information may improve the results of the ICH algorithm in some circumstances, the time required to generate $\Dbf$ would be substantial (Table~\ref{tab:dist_mat_times}). Nevertheless, the ICH algorithm is able to find very small resolving sets, in some cases less than half the size as those discovered using MINE (Table~\ref{tab:algo_adj_res_set}).

The discrepancy between resolving set sizes discovered with these two algorithms can be partially explained by the generality of MINE. The ICH algorithm focuses on the structure of a single network, taking advantage of specific relationships between nodes whereas MINE applies to an ensemble of graphs. In fact, its output is not a set of nodes but a number of nodes to choose from each community. This output can be used to generate highly likely resolving sets for any graph in the ensemble $SBM(n;C,P)$. This allows MINE to be very fast even for large networks. However, in order to accommodate the full ensemble $SBM(n;C,P)$, resulting resolving sets might be larger than those found by examining specific graph instances.

\begin{table}[h]
\scriptsize
\centering
\begin{tabular}{ccccc}
\toprule
Network & ICH & Greedy & Preorder & Random\\ 
\midrule
Political Blogs & $20182.13 \pm 886.16$ & $40.71 \pm 23.07$ & $67.08 \pm 17.88$ & $43.19 \pm 26.64$ \\ 
Political Books & $6073.3 \pm 288.33$ & $5.11 \pm 2.74$ & $51.79 \pm 2.1$ & $8.65 \pm 4.45$ \\ 
Karate Club & $4021.19 \pm 205.36$ & $1.43 \pm 0.71$ & $39.22 \pm 1.34$ & $2.03 \pm 1.23$ \\ 
\emph{David Copperfield} & $6362.8 \pm 312.41$ & $3.83 \pm 2.14$ & $39.39 \pm 1.95$ & $13.7 \pm 5.38$ \\ 
Primary School & $2709.34 \pm 310.22$ & $0.73 \pm 0.37$ & $56.65 \pm 2.06$ & $0.76 \pm 0.4$ \\ 
\bottomrule
\end{tabular}
\caption[LoF entry]{Time in seconds required to find resolving sets using only adjacency information ($n=10000$).}
\label{tab:adj_time}
\end{table}

\begin{table}[h]
\scriptsize
\centering
\begin{tabular}{ccccc}
\toprule
Network & ICH & Greedy & Preorder & Random\\ 
\midrule
Political Blogs & $184.66 \pm 0.91$ & $393.09 \pm 31.99$ & $345.87 \pm 27.15$ & $396.03 \pm 33.25$ \\ 
Political Books & $71.52 \pm 0.67$ & $118.41 \pm 9.39$ & $108.07 \pm 8.47$ & $126.34 \pm 11.75$ \\ 
Karate Club & $45.98 \pm 0.34$ & $66.8 \pm 4.7$ & $63.4 \pm 3.81$ & $69.98 \pm 6.19$ \\ 
\emph{David Copperfield} & $73.62 \pm 0.66$ & $123.73 \pm 9.64$ & $110.47 \pm 7.48$ & $151.41 \pm 14.98$ \\ 
Primary School & $32.3 \pm 0.46$ & $45.64 \pm 3.47$ & $42.53 \pm 2.53$ & $46.65 \pm 3.38$ \\ 
\bottomrule
\end{tabular}
\caption[LoF entry]{Resolving set sizes determined using only adjacency information ($n=10000$).}
\label{tab:adj_res_set}
\end{table}

The ideas underlying MINE can be used as the basis for a variety of other methods for finding resolving sets in $G \sim SBM(n;C,P)$. Several such approaches were implemented and tested alongside MINE and the ICH algorithm---see Table~\ref{tab:algo_descriptions}. 

For instance, the forward greedy search described in Algorithm~\ref{greedy_min_algo} can be used to iteratively add nodes to a growing set until it is resolving. Since this approach checks the resolvability of the set after the addition of each node, it requires substantially more time to find a suitable set than MINE. It is, however, far faster than the ICH algorithm, especially for large networks (Table~\ref{tab:adj_time}). In addition, resolving sets discovered in this way are smaller than those discovered by MINE but larger than those discovered with the ICH algorithm (Table~\ref{tab:adj_res_set}).

The greedy approach and MINE select a designated number of nodes uniformly at random within each community. Since the adjacencies of nodes from the same community come from the same distribution in the SBM, picking randomly within communities seems reasonable. However, because we are interested in minimizing the function $f$ in equations~(\ref{eq:monotonic_f})-(\ref{def:f(bfk)}), preemptively ordering nodes within communities based on this function may result in smaller resolving sets. Therefore, we also considered a preorder algorithm following this node selection strategy. Though this order is not updated as new nodes are added to a growing resolving set, initially sorting the nodes according to $f$ does slow the algorithm down (Table~\ref{tab:adj_time}). Taking advantage of small differences in nodes with respect to $f$ also usually results in the discovery of resolving sets slightly smaller than those discovered with a purely greedy method (Table~\ref{tab:adj_res_set}).

Finally, a fully randomized algorithm was tested. In particular, nodes were chosen uniformly at random and added to a growing set until a resolving set was generated. This strategy is comparable to but slightly slower than the greedy approach and results in slightly larger resolving sets (Table~\ref{tab:adj_res_set}). This similarity in performance is not surprising given how alike communities in the studied networks are with respect to size and adjacency probabilities. Consider, however, a graph $G \sim SBM(10000;C,P)$ with two communities of the same size and 

\[P = 
\begin{bmatrix}
    0.5 & 0.1 \\
    0.1 & 0.01
\end{bmatrix}.
\]

\noindent In this case, the greedy approach focuses on community one, taking $6.58 \pm 1.52$ seconds and finding resolving sets of size $85.17 \pm 6.81$ over 30 synthetic network instances and 50 replicates on each instance. The random strategy takes approximately the same number of nodes from each community. Since nodes from community two are not as helpful in uniquely distinguishing nodes, more time is taken ($20.86 \pm 5.53$ seconds) in generating larger resolving sets ($154.34 \pm 16.02$).

Summarizing, we have established very general regimes for which each of these algorithms work well. For small networks, the ICH algorithm is efficient enough to produce nearly minimal resolving sets in a reasonable amount of time. As network size increases and the ICH algorithm becomes prohibitively slow, new approaches for finding resolving sets must be applied. When $G \sim SBM(n;C,P)$, an iterative greedy method based on Algorithm~\ref{greedy_min_algo} and a strategy based on preordering the nodes of each community are fast and provide small resolving sets. For extremely large networks $G \sim SBM(n;C,P)$ with known parameters $C$ and $P$ and for which repeatedly accessing adjacency or distance information may be time consuming, MINE can very quickly determine a number of nodes to pick from each community to produce a resolving set with high probability. It is worth noting that the efficiency of MINE depends on the number of communities in the network. While this method is asymptotically faster than the ICH algorithm, it may be slower for certain values of $n$ and $c$.

\begin{algorithm}
\caption{Minimizing Indistinguishable Nodes via Expectation (MINE)}
\label{sbm_algo}
\begin{algorithmic}[1]
\STATE{Input: $c$, the number of dimensions}
\STATE{\hspace{1.04cm}$\alpha$, a threshold on values of $f$ to consider}
\STATE{Output: a point $\kbf \in \mathbb{N}_0^c$ such that $f(\mathbb{k}) \leq \alpha$ and $\sum_{i=1}^c \kbf_i$ is minimized.}
\STATE{$\xbf \gets ForwardGreedySearch(c, f, \alpha)$}
\STATE{$\xbf \gets BackwardGreedySearch(c, f, \alpha, \xbf)$}
\STATE{$\ybf \gets ((\sum_{i=1}^c \xbf_i)-1, 0, \dots, 0)$}
\WHILE{$\ybf \neq \{\}$}
\IF{$f(\ybf) \leq \alpha$}
\STATE{$\xbf \gets \ybf$}
\STATE{$\ybf \gets D(\ybf)$}
\ELSE
\STATE{$\ybf \gets NextPoint(c, \ybf)$}
\ENDIF
\ENDWHILE
\RETURN{$\xbf$}
\end{algorithmic}
\end{algorithm}

\begin{algorithm}
\caption{Forward Greedy Search}
\label{greedy_min_algo}
\begin{algorithmic}[1]
\STATE{Input: $c$, the number of dimensions}
\STATE{\hspace{1.04cm}$f$, the function to be minimized over $\mathbb{N}_0^c$}
\STATE{\hspace{1.04cm}$\alpha$, a threshold on values of $f$ to consider}
\STATE{Output: a point in $\mathbb{N}_0^c$ found via a greedy search to minimize $f$}
\STATE{$\xbf \gets (0, 0, \dots, 0)$}
\WHILE{$f(\xbf) > \alpha$}
\STATE{$\ybf \gets \xbf$}
\FOR{$i \in \{1, 2, \dots, c\}$}
\STATE{$\mathbf{e} = (0, 0, \dots, 0)$}
\STATE{$\mathbf{e}_i = 1$}
\IF{$f(\xbf+\mathbf{e}) < f(\ybf)$}
\STATE{$\ybf \gets \xbf+\mathbf{e}$}
\ENDIF
\ENDFOR
\STATE{$\xbf \gets \ybf$}
\ENDWHILE
\RETURN{$\xbf$}
\end{algorithmic}
\end{algorithm}

\begin{algorithm}
\caption{Backward Greedy Search}
\label{greedy_max_algo}
\begin{algorithmic}[1]
\STATE{Input: $c$, the number of dimensions}
\STATE{\hspace{1.04cm}$f$, the function to be maximized}
\STATE{\hspace{1.04cm}$\alpha$, a threshold on values of $f$ to consider}
\STATE{\hspace{1.04cm}$\xbf$, the starting point}
\STATE{Output: a point in $\mathbb{N}_0^c$ found via a backward greedy search to maximize $f$}
\STATE{$cont \gets true$}
\WHILE{$cont$}
\STATE{$cont \gets false$}
\STATE{$\ybf \gets \xbf$}
\FOR{$i \in \{1, 2, \dots, c\}$}
\STATE{$\mathbf{e} = (0, 0, \dots, 0)$}
\STATE{$\mathbf{e}_i = 1$}
\IF{$f(\ybf) < f(\xbf-\mathbf{e}) \leq \alpha$}
\STATE{$\ybf \gets \xbf-\mathbf{e}$}
\STATE{$cont \gets true$}
\ENDIF
\ENDFOR
\STATE{$\xbf \gets \ybf$}
\ENDWHILE
\RETURN{$\xbf$}
\end{algorithmic}
\end{algorithm}

\begin{algorithm}
\caption{Next Point ($\preceq$ Successor)}
\label{next_point}
\begin{algorithmic}[1]
\STATE{Input: $c$, the number of dimensions}
\STATE{\hspace{1.04cm}$\xbf \in S_h$}
\STATE{Output: the successor $\mathbf{n} \in S_h$ of $\xbf$ with respect to $\preceq$}
\IF{$\xbf_1 = \dots = \xbf_{c-1} = 0$}
\RETURN{$\{\}$ \textbackslash\textbackslash return the empty set if $\xbf$ has no successor in $S_h$}
\ENDIF
\STATE{$i = c-1$}
\WHILE{$\xbf_i=0$}
\STATE{$i \gets i-1$}
\ENDWHILE
\STATE{$\mathbf{n} \gets \xbf$}
\STATE{$\mathbf{n}_i \gets \xbf_i-1$}
\IF{$\xbf_c = 0$}
\STATE{$\mathbf{n}_{i+1} \gets \xbf_{i+1}+1$}
\ELSE
\STATE{$\mathbf{n}_{i+1} \gets \xbf_c + 1$}
\ENDIF
\STATE{$i \gets i+2$}
\WHILE{$i \leq c$}
\STATE{$\mathbf{n}_i \gets 0$}
\STATE{$i \gets i+1$}
\ENDWHILE
\RETURN{$\mathbf{n}$}
\end{algorithmic}
\end{algorithm}

\section{Conclusion and Future Directions}

The metric dimension of a graph is defined by pairwise distances between nodes. Dependencies among these values often make it difficult to characterize the metric dimension of random graph ensembles, including the SBM, precisely. We have shown, however, that adjacency information alone is enough to determine an upper bound on the metric dimension of any graph (Lemma~\ref{lem:Dmin2}). Furthermore, this bound is tight when $\diam(G) \leq 2$ (Corollary~\ref{cor:DA*}), a property that persists with high probability across a wide range of parameter values for the SBM. Taking advantage of this observation, we have described an algorithm (MINE) capable of finding small, highly likely resolving sets for a considerable fraction of graphs in the SBM ensemble with known parameters $C$ and $P$. This algorithm is efficient and has a much wider range of network sizes on which it is applicable as compared to the ICH algorithm.

Going forward, this algorithm might be used to discover resolving sets in networks with community structure in order to efficiently embed these kinds of networks in real space~\cite{tillquist2019low}. Such embeddings could be provided as input to machine learning and analysis algorithms to better study the processes underlying the formation of communities.

Furthermore, a similar approach to analyzing metric dimension may prove fruitful for other random graph ensembles such as those defined by a preferential attachment mechanism~\cite{barabasi1999emergence,price1976general,simon1955class}.

\section{Acknowledgements}

The authors thank Rafael Frongillo for his perceptive observations and useful suggestions.

This research was partially funded by the NSF IIS grant 1836914. The authors acknowledge the BioFrontiers Computing Core at the University of Colorado - Boulder for providing High-Performance Computing resources (fund\-ed by the NIH grant \\ 1S10OD012300), supported by BioFrontiers IT group.

\bibliographystyle{siamplain}

\end{document}